\documentclass[9pt,a4paper,draft]{article}
\usepackage{bbm}
\usepackage{pifont}
\usepackage{bbding}
\usepackage{amsmath}
\usepackage{mathrsfs}
\usepackage{amsfonts}
\usepackage{amssymb}
\usepackage{amsfonts,amssymb,amsmath,indentfirst,amsthm}
\usepackage{epsfig}
\usepackage[numbers,sort&compress]{natbib}
\usepackage{mathtools}
\def\ess~inf{\mathop{\rm ess~inf}}

\numberwithin{equation}{section}

\newenvironment{key words}{\emph{\texttt{Keywords}}\mbox{  }}{ }
\newtheorem{theorem}{Theorem}[section]
\newtheorem{lemma}[theorem]{Lemma}

\renewenvironment{proof}{\noindent{\textbf{Proof.}}}{\hfill$\Box$}
\theoremstyle{remark}

\theoremstyle{plain}

\makeatletter

\newcommand{\Rmnum}[1]{\expandafter\@slowromancap\romannumeral #1@}
\makeatother

\begin{document}
\setlength{\baselineskip}{1.2\baselineskip}
\title
{\textbf{A study on a class of generalized Schr\"{o}dinger operators} \thanks{This work is supported by Natural Science Foundation of China (No.11601427);
China Postdoctoral Science Foundation (No.2017M613193);  Natural Science Basic Research Plan in Shaanxi Province of China (No.2017JQ1009).}}

\author{{Wenjuan Li \qquad \qquad Huiju Wang }  \\
(\small{School of Mathematics and Statistics, Northwestern Polytechnical University,
 Xi'an, 710129, China})\\    }
\date{}
 \maketitle
{\bf Abstract:}\  In this paper, we consider the pointwise convergence for a class of generalized Schr\"{o}dinger operators with suitable perturbations, and convergence rate for a class of generalized Schr\"{o}dinger operators with polynomial growth. We show that the pointwise convergence results remain valid for a class of generalized Schr\"{o}dinger operators under small perturbations. As applications, we obtain the sharp convergence result for Boussinesq operator and Beam operator in $\mathbb{R}^2$. Moreover, the convergence result for a class of non-elliptic Schr\"{o}dinger operators with finite-type perturbations is built. Furthermore, we proved that the convergence rate for a class of generalized Schr\"{o}dinger operators with polynomial growth depends only on the growth condition of their phase functions. This result can be applied to all previously mentioned operators, and more operators.



\section{\textbf{Introduction}\label{Section 1}}
Consider the generalized Schr\"{o}dinger equation
\begin{equation}\label{Eq1}
\begin{cases}
 \partial_{t}u(x,t)-iP(D)u(x,t) =0 \:\:\:\ x \in \mathbb{R}^{n}, t \in \mathbb{R}^{+},\\
u(x,0)=f \\
\end{cases}
\end{equation}
where $D=\frac{1}{i}(\frac{\partial}{\partial x_{1}},\frac{\partial}{\partial x_{2}},...,\frac{\partial}{\partial x_{n}})$, $P(\xi)$ is a real continuous function defined on $\mathbb{R}^{n}$, $P(D)$ is defined via its real symbol
\[P(D)f(x)= \int_{\mathbb{R}^{n}}{e^{ix \cdot \xi}P(\xi)\hat{f}(\xi)d\xi}.\]
 The solution of (\ref{Eq1}) can be formally written as
\begin{equation}
e^{itP(D)}f(x):= \int_{\mathbb{R}^{n}}{e^{ix \cdot \xi +itP(\xi)} \hat{f} (\xi)d\xi },
\end{equation}
where $\hat{f}(\xi)$ denotes the Fourier transform of $f$.

The convergence problem, that is. to determine the optimal $s$ for which
\begin{equation}
\mathop{lim}_{t \rightarrow 0^{+}} e^{itP(D)}f(x) = f(x)
\end{equation}
almost everywhere whenever $f \in H^{s}(\mathbb{R}^{n})$, has been widely studied since the first work by Carleson (\cite{C}), see \cite{L}, \cite{V}, \cite{SS}, \cite{S}, \cite{Miao} and references therein. Sharp results were derived in some cases, such as  the elliptic case (\cite{DGL,DZ}, when $n \ge 1$, $P(\xi)=|\xi|^{2}$); the non-elliptic case (\cite{KPV}, when $n \ge 1$, $P(\xi)=\xi_{1}^{2}-\xi_{2}^{2} \pm \cdot \cdot \cdot \pm \xi_{n}^{2}$) and the fractional case (\cite{CK}, when $n \ge 1$ and $P(\xi)= |\xi|^{\alpha}$, $\alpha >1$).

%%%%%%%%%%%%%%%%%%%%%%%%%%%%%%%%%%%%%%%%%%%%%%%%%%%%%%%%%%%%%%%%%%%%%%%%%%%%%%%%%%%%%%%%%%%%%%%%%%%%%%%%%%%%%%%%%%%%%%%%%%%%%%%%
In this paper, we firstly consider the convergence problem for a class of generalized Schr\"{o}dinger operators with small perturbations. We first establish the following  general results.

\begin{theorem}\label{theorem2.1}
If there exist a real continuous function $Q(\xi)$ and a real number $s_{0} >0$ such that
\begin{equation}\label{Eq2.1.1}
|P(\xi )-Q(\xi)| \lesssim 1, |\xi| \rightarrow +\infty,
\end{equation}
and for any $s > s_{0}$,
\begin{equation}\label{Eq2.1.2}
\biggl\|\mathop{sup}_{0<t<1} |e^{itQ(D)}f|\biggl\|_{L^{p}(B(0,1))} \lesssim \big\|f\big\|_{H^{s}(\mathbb{R}^{n})},  \:\ p \ge 1,
\end{equation}
then for all $s > s_{0}$ and $f \in H^{s}(\mathbb{R}^{n})$,
\begin{equation}\label{Eq2.1.3}
\biggl\|\mathop{sup}_{0<t<1} |e^{itP(D)}f|\biggl\|_{L^{p}(B(0,1))} \lesssim \big\|f\big\|_{H^{s}(\mathbb{R}^{n})},  \:\ p \ge 1.
\end{equation}
\end{theorem}

Theorem \ref{theorem2.1} implies the Equivalence between the convergence property of operators with small perturbations. Theorem \ref{theorem2.1} is quite general and can be applied to a wide class of operators. In particular, we concentrate ourselves on $n=2$, and consider the Boussinesq operator defined by
\[P_B(\xi) = |\xi|\sqrt{1+|\xi|^{2}},\]
and obtain the following almost sharp result:

\begin{theorem}\label{theorem2.2}
(1) For each $s > 1/3$, if $f \in H^{s}(\mathbb{R}^{2})$, then
\begin{equation}\label{Eq2.2.1}
\biggl\|\mathop{sup}_{0<t<1} |e^{itP_B(D)}f|\biggl\|_{L^{3}(B(0,1))} \lesssim \|f\|_{H^{s}(\mathbb{R}^{2})}.
\end{equation}

(2) For each $s< \frac{1}{3}$, there exists $f \in L^{2}(\mathbb{R}^{n})$ and $\hat{f}$  supported in the annulus $\{\xi \in \mathbb{R}^{2}: |\xi| \sim R\}$, such that
\begin{equation}\label{Eq2.2.2}
\mathop{lim}_{R \rightarrow +\infty{}} \frac {R^{-s} \|\mathop{sup}_{0<t<1} |e^{itP_B(D)}f|\|_{L^{1}(B(0,1))}}{\|f\|_{L^{2}}} = +\infty.
\end{equation}
\end{theorem}

By the same method, we can prove that the results in Theorem \ref{theorem2.2} also hold for operators such as Beam operator $P(\xi)= \sqrt{1+|\xi|^{4}}$. But we omit its proof here.
%%%%%%%%%%%%%%%%%%%%%%%%%%%%%%%%%%%%%%%%%%%%%%%%%%%%%%%%%%%%%%%%%%%%%%%%%%%%%%%%%%%%%%%%%%%%%%%%%%%%%%%%%%%%%%%%%%%%%%%%%%%%%%%%%%%

Recently, Buschenhenke, M\"{u}ller and Vargas \cite{BMV1, BMV2} studied Fourier restriction estimate for finite-type perturbations of the hyperbolic paraboloid. We are also curious about how "finite-type perturbations" works in the corresponding generalized Schr\"{o}dinger equation. Next, we concentrate ourselves on $n=2$, $m \ge 1$. Consider a class of operators with phase function
\[P_{m}(\xi) = \xi_{1}\xi_{2} + h_m(\xi_{1}),\]
where $h_{m}(\xi_{1})=\frac{1}{m}\xi_{1}^{m}$ when $m \in \mathbb{N}^{+}$. In this case, the corresponding equations are higher order dispersive equations, see \cite{GC} and its references for more information. when $1 < m <2$, $h_{m}(\xi_{1})=\frac{1}{m}|\xi_{1}|^{m}$, the corresponding equations are non-elliptic Schr\"{o}dinger equations with fractional order perturbations. We obtained the following result.

\begin{theorem}\label{theorem2.3}
(1) For each $m \in \mathbb{N}^{+}$, $s > 1/2$, if $f \in H^{s}(\mathbb{R}^{2})$, then
\begin{equation}\label{Eq2.3.1}
\biggl\|\mathop{sup}_{0<t<1} |e^{itP_{m}(D)}f|\biggl\|_{L^{2}(B(0,1))} \lesssim \|f\|_{H^{s}(\mathbb{R}^{2})}.
\end{equation}

(2) The similarly convergence results hold for $1<m<2$ and $s > \frac{1}{2}$. In particular, for $s< \frac{1}{2}$, there exists $f \in L^{2}(\mathbb{R}^{n})$ and $\hat{f}$  supported in the annulus $\{\xi \in \mathbb{R}^{2}: |\xi| \sim R\}$, such that
\begin{equation}\label{Eq2.3.2}
\mathop{lim}_{R \rightarrow +\infty{}} \frac {R^{-s} \big\|\mathop{sup}_{0<t<1} |e^{itP_{m}(D)}f|\big\|_{L^{1}(B(0,1))}}{\|f\|_{L^{2}}} = +\infty.
\end{equation}
\end{theorem}

By \cite{RVV}, $s > \frac{1}{2}$ is likely sharp for the convergence result to hold in the non-elliptic case up to the end point. Theorem \ref{theorem2.3} implies that the "finite-type perturbations" does not change the convergence result for $s > \frac{1}{2}$. Moreover, for $1 < m <2$, our convergence result is sharp up to the end point.

%%%%%%%%%%%%%%%%%%%%%%%%%%%%%%%%%%%%%%%%%%%%%%%%%%%%%%%%%%%%%%%%%%%%%%%%%%%%%%%%%%%%%%%%%%%%%%%%%%%%%%%%%%%%%%%%%%%%%%%%%%%%%%%%
Furthermore, it is interesting to seek the convergence speed of $e^{itP(D)}f(x)$ as $t$ tends to $0$ if $f$ has more regularity. The problem is, suppose that $e^{itP(D)}f(x)$ converge to $f$ for $f \in H^{s}(\mathbb{R}^{n})$ as $t$ tends to $0$, whether or not it is possible that, for $f \in H^{s + \delta}(\mathbb{R}^{n})$, $\delta \ge 0$,
\begin{equation}
e^{itP(D)}f(x) - f(x) = o(t^{\theta(\delta)})
\end{equation}
almost everywhere for some $\theta(\delta) \ge 0$? Cao, Fan and Wang \cite{CFW} proved this property in the elliptic case when $n \ge 1$, $P(\xi) = |\xi|^{2}$, $\theta(\delta) = \frac{\delta}{2}$, $0 \le \delta <2$, and in the fractional case when $n = 1$, $P(\xi)=|\xi|^{\alpha}, \alpha >1$, $\theta(\delta) = \frac{\delta}{\alpha}$, $0 \le \delta < \alpha$.

In this paper, we obtain the convergence rate for a class of Schr\"{o}dinger operators with polynomial growth:

\begin{theorem}\label{theorem2.4}
If there exist $m > 0$, $s_{0} >0$ such that
\begin{equation}\label{Eq2.4.1}
|P(\xi)| \lesssim |\xi|^{m}, |\xi| \rightarrow +\infty,
\end{equation}
and for each $s > s_{0}$,
\begin{equation}\label{Eq2.4.2}
\biggl\|\mathop{sup}_{0<t<1} |e^{itP(D)}f|\biggl\|_{L^{p}(B(0,1))} \lesssim \|f\|_{H^{s}(\mathbb{R}^{n})},  \:\ p \ge 1,
\end{equation}
then for all $f \in H^{s+\delta}(\mathbb{R}^{n})$, $0 \le \delta <m$,
\begin{equation}\label{Eq2.4.3}
e^{itP(D)}f(x) - f(x) = o(t^{\delta /m}), \:\ a.e. \text{\quad as \quad} t \rightarrow 0^{+}.
\end{equation}
\end{theorem}

Note that the convergence rate in Theorem \ref{theorem2.4} depends on the growth condition of the phase function, but independent of its gradient and the dimension of the spatial space. Theorem \ref{theorem2.4} is quite general and can be applied to a wide class of operators, such as the non-elliptic Schr\"{o}dinger operators ($P(\xi)=\xi_{1}^{2}-\xi_{2}^{2} \pm \cdot \cdot \cdot \pm \xi_{n}^{2}$), the fractional Schr\"{o}dinger operators ($P(\xi)= |\xi|^{\alpha}$, $\alpha >1$) and the Boussinesq operator ($P(\xi)=|\xi|\sqrt{1+|\xi|^{2}}$). It also generalized the previous result of \cite{CFW}.

%%%%%%%%%%%%%%%%%%%%%%%%%%%%%%%%%%%%%%%%%%%%%%%%%%%%%%%%%%%%%%%%%%%%%%%%%%%%%%%%%%%%%%%%%%%%%%%%%%%%%%%%%%%%%%%%%%%%%%%%%%%%

\section{Proof of Theorem \ref{theorem2.1}}\label{Section 2}

\textbf{Proof of Theorem \ref{theorem2.1}.}
 In order to show (\ref{Eq2.1.3}), we decompose $f$ as
\[f=\sum_{k=0}^{\infty}{f_{k}},\]
where $supp \hat{f_{0}} \subset B(0,1)$, $supp \hat{f_{k}} \subset \{\xi: |\xi| \sim 2^{k}\}, k \ge 1$. Then we have
\begin{equation}\label{Eq2.1.4}
\biggl\|\mathop{sup}_{0<t<1} |e^{itP(D)}f|\biggl\|_{L^{p}(B(0,1))} \le \sum_{k=0}^{\infty}{\biggl\|\mathop{sup}_{0<t<1} |e^{itP(D)}f_{k}|\biggl\|_{L^{p}(B(0,1))}}.
\end{equation}

For $k \lesssim 1$, since for each $x \in B(0,1)$,
\[\biggl|e^{itP(D)}f_{k}(x)\biggl| \lesssim \|f_{k}\|_{L^{2}(\mathbb{R}^{n})},\]
it is obvious that
\begin{equation}\label{Eq2.1.5}
\biggl\|\mathop{sup}_{0<t<1} |e^{itP(D)}f_{k}|\big\|_{L^{p}(B(0,1))} \lesssim \|f\|_{H^{s}(\mathbb{R}^{n})}.
\end{equation}

For $k \gg 1$, by Taylor's formula, for each $k$,
\begin{equation}\label{Eq2.1.6}
\biggl|e^{itP(D)}(f_{k})-e^{itQ(D)}(f_{k})\biggl| \le \sum_{j=1}^{\infty}{\frac{t^{j}}{j!} \biggl|\int_{\mathbb{R}^{n}}{e^{ix\cdot\xi + itQ(\xi)}[P(\xi)-Q(\xi)]^{j}\hat{f_{k}}}(\xi)d\xi\biggl|  }.
\end{equation}
It is obvious that
\begin{align}\label{Eq2.1.7}
\biggl\|\mathop{sup}_{0<t<1} |e^{itP(D)}f_{k}| \biggl\|_{L^{p}(B(0,1))} &\le \biggl\|\mathop{sup}_{0<t<1} |e^{itP(D)}(f_{k})-e^{itQ(D)}{(f_{k})}| \biggl\|_{L^{p}(B(0,1))} \nonumber\\
& \:\ + \biggl\|\mathop{sup}_{0<t<1} |e^{itQ(D)}f_{k}| \biggl\|_{L^{p}(B(0,1))}.
\end{align}

For $\forall\epsilon>0$, from (\ref{Eq2.1.2}), for each $g$ whose Fourier transform is supported in $\{\xi: |\xi| \sim 2^{k}\}$, we have
\begin{equation}\label{Eq2.1.8}
\biggl\|\mathop{sup}_{0 < t<1} |e^{itQ(D)}g|\biggl\|_{L^{p}(B(0,1))} \lesssim 2^{(s_{0}+\frac{\epsilon}{2})k}\|g\|_{L^{2}(\mathbb{R}^{n})}.
\end{equation}
Let $s_1=s_0+\epsilon$, then
\begin{equation}\label{Eq2.1.9}
\biggl\|\mathop{sup}_{0 < t<1} |e^{itQ(D)}f_{k}|\biggl\|_{L^{p}(B(0,1))} \lesssim 2^{-\frac{k\epsilon}{2}}\|f\|_{H^{s_1}(\mathbb{R}^{n})}.
\end{equation}
Inequalities (\ref{Eq2.1.6}), (\ref{Eq2.1.8}) and (\ref{Eq2.1.1}) imply
\begin{align}\label{Eq2.1.10}
&\biggl\|\mathop{sup}_{0<t<1} |e^{itP(D)}(f_{k})-e^{itQ(D)}{(f_{k})}| \biggl\|_{L^{p}(B(0,1))}  \nonumber\\
&\le \sum_{j=1}^{\infty}{\frac{1}{j!} \biggl\|\mathop{sup}_{0<t<1}\biggl|\int_{ \mathbb{R}^{n}}{e^{ix\cdot\xi + itQ(\xi)}[P(\xi)-Q(\xi)]^{j}\hat{f_{k}}}(\xi)d\xi\biggl|\biggl\|_{L^{p}(B(0,1))}} \nonumber\\
&\le \sum_{j=1}^{\infty}{\frac{2^{k(s_{0}+\frac{\epsilon}{2})}}{j!} \|[P(\xi)-Q(\xi)]^{j}\hat{f_{k}}}(\xi)\|_{L^{2}(\mathbb{R}^{n})} \nonumber\\
&\le \sum_{j=1}^{\infty}{\frac{C^{j}2^{-\frac{k\epsilon}{2}}}{j!} \|f}\|_{H^{s_1}(\mathbb{R}^{n})} \nonumber\\
&\lesssim 2^{-\frac{k\epsilon}{2}}\|f\|_{H^{s_1}(\mathbb{R}^{n})}.
\end{align}
Inequalities (\ref{Eq2.1.7}), (\ref{Eq2.1.9}) and (\ref{Eq2.1.10}) yield for $k \gg 1$,
\begin{align}\label{Eq2.1.11}
\biggl\|\mathop{sup}_{0<t<1} |e^{itP(D)}f_{k}|\biggl\|_{L^{p}(B(0,1))} &\lesssim 2^{-\frac{\epsilon k}{2}}\|f\|_{H^{s_1}(\mathbb{R}^{n})}.
\end{align}

Combing  (\ref{Eq2.1.4}), (\ref{Eq2.1.5}) and (\ref{Eq2.1.11}),  inequality (\ref{Eq2.1.3}) holds true for $s_1$. By the arbitrariness of $\epsilon$, in fact, we can get for any $s>s_0$, inequality (\ref{Eq2.1.3}) remains true.

%%%%%%%%%%%%%%%%%%%%%%%%%%%%%%%%%%%%%%%%%%%%%%%%%%%%%%%%%%%%%%%%%%%%%%%%%%%%%%%%%%%%%%%%%%%%%%%%%%%%%%%%%%%%%%%%%%%%%%%%%%%%%%%%%%%%

\section{Proof of Theorem \ref{theorem2.2}}\label{Section 3}
\textbf{Proof of Theorem \ref{theorem2.2}.} (1) Inequality (\ref{Eq2.2.1}) follows directly form Theorem \ref{theorem2.1} and the following convergence result for Schr\"{o}dinger operator (\cite{DGL}).

\begin{theorem}\label{lemma3.2}(\cite{DGL}) For any $s>1/3$, the following bounds hold: for any function $\hat{f}\in H^s(\mathbb{R}^2)$,
\begin{equation*}
\biggl\|\mathop{sup}_{0<t<1}|e^{it\Delta}f(x)|\biggl\|_{L^3(B(0,1))} \leq C_s\|f\|_{H^s}.
\end{equation*}
\end{theorem}

(2) In \cite{B}, Bourgain actually showed that there exists $f$,
\[\hat{f}(\xi) = \chi_{A_{R}}(\xi),\]
where $A_{R}$ is the subset of $\{\xi \in \mathbb{R}^{2}: |\xi| \sim R\}$ defined by
\[A_{R}= \bigcup_{l \in \mathbb{N}^{+}, l \sim R^{1/3}}A_{R,l},\]
\[A_{R,l}= [R-R^{1/2}, R+R^{1/2}] \times [R^{2/3}l, R^{2/3}l+1].\]
And there exists a set $S$ with positive measure such that for each $x \in S$, there exists $t, |t| \le R^{-1}$,
\begin{equation}
|e^{it\Delta}f(x)| \ge R^{3/4}.
\end{equation}
Hence,
\begin{equation}\label{Eq2.2.3}
\mathop{sup}_{0<t<R^{-1}} |e^{it\Delta}f(x)| \ge R^{3/4}.
\end{equation}
By Taylor expansion,
\begin{align}\label{Eq2.2.4}
\mathop{sup}_{0<t<R^{-1}} \biggl|e^{it\Delta}f(x)\biggl| &\le \mathop{sup}_{0<t<R^{-1}} \biggl|e^{it\Delta}f(x)-e^{itP_B(D)}f(x)\biggl| + \mathop{sup}_{0<t<R^{-1}} \biggl|e^{itP_B(D)}f(x)\biggl| \nonumber\\
&\le \sum_{j=1}^{\infty}{\frac{R^{-j}}{j!} \int_{\mathbb{R}^{n}}{|\hat{f}}(\xi)|d\xi  } + \mathop{sup}_{0<t<1} |e^{itP_B(D)}f(x)| \nonumber\\
&\le R^{-1}R^{1/3}R^{1/2}+  \mathop{sup}_{0<t<1} \biggl|e^{itP_B(D)}f(x)\biggl|.
\end{align}
Inequalities (\ref{Eq2.2.3}) and (\ref{Eq2.2.4}) imply
\begin{equation}\label{Eq2.2.5}
 \big\|\mathop{sup}_{0<t<1} |e^{itP_B(D)}f|\big\|_{L^{1}(B(0,1))} \gtrsim R^{3/4},
\end{equation}
which implies (\ref{Eq2.2.2}).

%%%%%%%%%%%%%%%%%%%%%%%%%%%%%%%%%%%%%%%%%%%%%%%%%%%%%%%%%%%%%%%%%%%%%%%%%%%%%%%%%%%%%%%%%%%%%%%%%%%%%%%%%%%%%%%%%%%%%%%%%%%%%

\section{Proof of Theorem \ref{theorem2.3}}\label{Section 4}
We first prove the following Lemma \ref{lemma3.1}.
\begin{lemma}\label{lemma3.1}
Assume that $g$ is a Schwartz function whose Fourier transform is supported away from $0$. Then
\begin{align}\label{Eq2.3.3}
\biggl\|\mathop{sup}_{0<t<1} |e^{itP_{m}(D)}g|\biggl\|_{L^{2}(B(0,1))}
\le \|g\|_{L^{2}(\mathbb{R}^{2})} + \biggl(\int{\frac{|P_{m}(\xi)|^{2}}{|\nabla P_{m}(\xi)|} |\hat{g}(\xi)|^{2} d\xi}\biggl)^{\frac{1}{4}}\biggl(\int{\frac{1}{|\nabla P_{m}(\xi)|} |\hat{g}(\xi)|^{2} d\xi}\biggl)^{\frac{1}{4}}.
\end{align}
\end{lemma}
\begin{proof}
For each $x \in B(0,1)$,
\begin{align}
\mathop{sup}_{0<t<1}|e^{itP_{m}(D)}g(x)|^{2} &\le|g(x)|^{2} + \biggl(\int_{0}^{1}|\int_{\mathbb{R}^{2}}{e^{ix \cdot \xi + itP_{m}(\xi)}}\hat{g}(\xi)d\xi|^{2}dt\biggl)^{\frac{1}{2}}\\
&\times\biggl(\int_{0}^{1}|\int_{\mathbb{R}^{2}}{e^{ix \cdot \xi+itP_{m}(\xi)}}P_{m}(\xi)\hat{g}(\xi)d\xi|^{2}dt\biggl)^{\frac{1}{2}}.
\end{align}
By H\"{o}lder's inequality,
\begin{align}\label{Eq2.3.4}
\biggl\|\mathop{sup}_{0<t<1} |e^{itP_{m}(D)}g|\biggl\|_{L^{2}(B(0,1))} &\le \|g\|_{L^{2}(\mathbb{R}^{2})} + \biggl(\int_{B(0,1)}{\int_{0}^{1}|\int_{\mathbb{R}^{2}}{e^{ix \cdot \xi + itP_{m}(\xi)}}\hat{g}(\xi)d\xi|^{2}}dtdx\biggl)^{\frac{1}{4}} \nonumber\\
& \:\ \times \biggl(\int_{B(0,1)}{\int_{0}^{1}|\int_{\mathbb{R}^{2}}{e^{ix \cdot \xi + itP_{m}(\xi)}}P_{m}(\xi)\hat{g}(\xi)d\xi|^{2}}dtdx\biggl)^{\frac{1}{4}}.
\end{align}
By Theorem 4.1 in \cite{KPV},
\begin{align}\label{Eq2.3.5}
&\int_{B(0,1)}{\int_{0}^{1}\biggl|\int_{\mathbb{R}^{2}}{e^{ix_{1}  \xi_{1} + ix_{2}  \xi_{2} + itP_{m}(\xi_{1},\xi_{2})}}\hat{g}(\xi_{1}, \xi_{2})d\xi_{1}d\xi_{2}\biggl|^{2}}dtdx \nonumber\\
&\lesssim \int_{B(0,1)}{\int_{0}^{1}\biggl|\int_{\mathbb{R}^{2}}{e^{i\eta_{1}(x_{1} + x_{2}) + i\eta_{2} (x_{1}-x_{2}) + itP_{m}(\eta_{1}+\eta_{2}, \eta_{1}-\eta_{2})}}\hat{g}(\eta_{1}+\eta_{2}, \eta_{1}-\eta_{2})d\eta_{1}d\eta_{2}\biggl|^{2}}dtdx \nonumber\\
&\lesssim \int_{B(0,2)}{\int_{0}^{1}\biggl|\int_{\mathbb{R}^{2}}{e^{i\eta_{1}y_{1} + i\eta_{2}y_{2} + itP_{m}(\eta_{1}+\eta_{2}, \eta_{1}-\eta_{2})}}\hat{g}(\eta_{1}+\eta_{2}, \eta_{1}-\eta_{2})d\eta_{1}d\eta_{2}\biggl|^{2}}dtdy \nonumber\\
&\lesssim \int_{\mathbb{R}^{2}}{\frac{|\hat{g}(\eta_{1}+\eta_{2}, \eta_{1}-\eta_{2})|^{2}}{|\nabla P_{m}(\eta_{1}+\eta_{2}, \eta_{1}-\eta_{2})|}d\eta_{1}d\eta_{2} } \nonumber\\
&\lesssim \int_{\mathbb{R}^{2}}{\frac{|\hat{g}(\xi_{1}, \xi_{2})|^{2}}{|\nabla P_{m}(\xi_{1}, \xi_{2})|}d\xi_{1}d\xi_{2}} .
\end{align}
For the same reason,
\begin{align}\label{Eq2.3.6}
&\int_{B(0,1)}{\int_{0}^{1}|\int_{\mathbb{R}^{2}}{e^{ix_{1}  \xi_{1} + ix_{2}  \xi_{2} + itP_{m}(\xi_{1},\xi_{2})}}P_{m}(\xi_{1},\xi_{2})\hat{g}(\xi_{1}, \xi_{2})d\xi_{1}d\xi_{2}|^{2}}dtdx \nonumber\\
&\lesssim \int_{\mathbb{R}^{2}}{\frac{|P_{m}(\xi_{1},\xi_{2})|^{2}|\hat{g}(\xi_{1}, \xi_{2})|^{2}}{|\nabla P_{m}| (\xi_{1}, \xi_{2})}d\xi_{1}d\xi_{2}}.
\end{align}
Inequality (\ref{Eq2.3.3}) follows from (\ref{Eq2.3.4}), (\ref{Eq2.3.5}) and (\ref{Eq2.3.6}).
\end{proof}

\textbf{Proof of Theorem \ref{theorem2.3}.}
(1) We decompose $f$ as
\[f=\sum_{k=0}^{\infty}{f_{k}},\]
where supp$\hat{f_{0}} \subset B(0,1)$,  supp$ \hat{f_{k}} \subset \{\xi: |\xi| \sim 2^{k}\}, k \ge 1$. Then we have
\begin{equation}\label{Eq2.3.7}
\biggl\|\mathop{sup}_{0<t<1} |e^{itP_{m}(D)}f|\biggl\|_{L^{2}(B(0,1))} \le \sum_{k=0}^{\infty}{\biggl\|\mathop{sup}_{0<t<1} |e^{itP_{m}(D)}f_{k}|\biggl\|_{L^{2}(B(0,1))}}.
\end{equation}

For $k \lesssim 1$, since for each $x \in B(0,1)$,
$$\biggl|e^{itP_{m}(D)}f_{k}(x)\biggl| \lesssim \|f_{k}\|_{L^{2}(\mathbb{R}^{2})},$$
it is obvious that
\begin{equation}
\biggl\|\mathop{sup}_{0<t<1} |e^{itP_{m}(D)}f_{k}|\biggl\|_{L^{2}(B(0,1))} \lesssim \|f\|_{H^{s}(\mathbb{R}^{2})}.
\end{equation}

For $k \gg 1$, we decompose each $f_{k}$ as
\[f_{k} = \sum_{j=1}^{3}{f_{k,j}},\]
where supp$ \widehat{f_{k,j}} \subset A_{k,j}, j=1,2,3$,
\[A_{k,1}=\{\xi: |\xi| \sim 2^{k}, |\xi_{2}| \gg |\xi_{1}|^{m-1}\},\]
\[A_{k,2}=\{\xi: |\xi| \sim 2^{k}, |\xi_{2}| \sim |\xi_{1}|^{m-1}\},\]
\[A_{k,3}=\{\xi: |\xi| \sim 2^{k}, |\xi_{2}| \ll |\xi_{1}|^{m-1}\},\]
then
\begin{equation}
\biggl\|\mathop{sup}_{0<t<1} |e^{itP_{m}(D)}f_{k}|\biggl\|_{L^{2}(B(0,1))} \le \sum_{j=1}^{3}{\biggl\|\mathop{sup}_{0<t<1} |e^{itP_{m}(D)}f_{k,j}|\biggl\|_{L^{2}(B(0,1))}}.
\end{equation}

By Lemma \ref{lemma3.1},
\begin{align}\label{Eq2.3.8}
&\biggl\|\mathop{sup}_{0<t<1} |e^{itP_{m}(D)}f_{k,1}|\biggl\|_{L^{2}(B(0,1))}  \nonumber\\
&\le \|f_{k,1}\|_{L^{2}(\mathbb{R}^{2})} + \biggl(\int{\frac{|\xi_{1}\xi_{2}+\frac{1}{m}\xi_{1}^{m}|^{2}}{|\xi_{2}+\xi_{1}^{m-1}|+|\xi_{1}|} |\widehat{f_{k,1}}(\xi_{1}, \xi_{2})|^{2} d\xi_{1}d\xi_{2}}\biggl)^{\frac{1}{4}} \nonumber\\
& \:\ \times \biggl(\int{\frac{1}{|\xi_{2}+\xi_{1}^{m-1}|+|\xi_{1}|} |\widehat{f_{k,1}}(\xi_{1}, \xi_{2})|^{2} d\xi_{1}d\xi_{2}}\biggl)^{\frac{1}{4}} \nonumber\\
&\lesssim min\{2^{\frac{k}{2}}, 2^{\frac{k}{2(m-1)}}\} \|f_{k,1}\|_{L^{2}(\mathbb{R}^{2})} \nonumber\\
&\le 2^{(-s+\frac{1}{2})k} \|f\|_{H^{s}(\mathbb{R}^{2})}.
\end{align}
Analogously,
\begin{align}\label{Eq2.3.9}
&\biggl\|\mathop{sup}_{0<t<1} |e^{itP_{m}(D)}f_{k,2}|\biggl\|_{L^{2}(B(0,1))}  \nonumber\\
&\le \|f_{k,2}\|_{L^{2}(\mathbb{R}^{2})} + \biggl(\int{\frac{|\xi_{1}\xi_{2}+\frac{1}{m}\xi_{1}^{m}|^{2}}{|\xi_{2}+\xi_{1}^{m-1}|+|\xi_{1}|} |\widehat{f_{k,2}}(\xi_{1}, \xi_{2})|^{2} d\xi_{1}d\xi_{2}}\biggl)^{\frac{1}{4}} \nonumber\\
& \:\ \times \biggl(\int{\frac{1}{|\xi_{2}+\xi_{1}^{m-1}|+|\xi_{1}|} |\widehat{f_{k,2}}(\xi_{1}, \xi_{2})|^{2} d\xi_{1}d\xi_{2}}\biggl)^{\frac{1}{4}} \nonumber\\
&\lesssim 2^{\frac{k}{2}} \|f_{k,2}\|_{L^{2}(\mathbb{R}^{2})} \nonumber\\
&\le 2^{(-s+\frac{1}{2})k} \|f\|_{H^{s}(\mathbb{R}^{2})}.
\end{align}

In order to deal with $f_{k,3}$, we further decompose $A_{k,3} = \bigcup_{l=1}^{k} A_{k,3}^{l}$, where
\begin{equation}\label{Eq2.3.10}
A_{k,3}^{l}=\{\xi: |\xi| \sim 2^{k}, 2^{l-1} \le |\xi_{1}|<2^{l}, |\xi_{2}| \ll |\xi_{1}|^{m-1}\},
\end{equation}
and
\[f_{k,3} = \sum_{l=1}^{k}{f_{k,3}^{l}},\]
such that supp$ \widehat{f_{k,3}^{l}} \subset A_{k,3}^{l}$, $1 \le l \le k$. Then for each $f_{k,3}^{l}$,
\begin{align}\label{Eq2.3.11}
&\biggl\|\mathop{sup}_{0<t<1} |e^{itP_{m}(D)}f_{k,3}^{l}|\biggl\|_{L^{2}(B(0,1))}  \nonumber\\
&\le \|f_{k,3}^{l}\|_{L^{2}(\mathbb{R}^{2})} + \biggl(\int{\frac{|\xi_{1}\xi_{2}+\frac{1}{m}\xi_{1}^{m}|^{2}}{|\xi_{2}+\xi_{1}^{m-1}|+|\xi_{1}|} |\widehat{f_{k,3}^{l}}(\xi_{1}, \xi_{2})|^{2} d\xi_{1}d\xi_{2}}\biggl)^{\frac{1}{4}} \nonumber\\
& \:\ \times \biggl(\int{\frac{1}{|\xi_{2}+\xi_{1}^{m-1}|+|\xi_{1}|} |\widehat{f_{k,3}^{l}}(\xi_{1}, \xi_{2})|^{2} d\xi_{1}d\xi_{2}}\biggl)^{\frac{1}{4}} \nonumber\\
&\lesssim 2^{\frac{l}{2}} \|f_{k,3}^{l}\|_{L^{2}(\mathbb{R}^{2})}.
\end{align}
Due to (\ref{Eq2.3.10}) and (\ref{Eq2.3.11}), we have
\begin{align}\label{Eq2.3.12}
\biggl\|\mathop{sup}_{0<t<1} |e^{itP_{m}(D)}f_{k,3}|\biggl\|_{L^{2}(B(0,1))} &\le \sum_{l=1}^{k}{\biggl\|\mathop{sup}_{0<t<1} |e^{itP_{m}(D)}f_{k,3}^{l}|\biggl\|_{L^{2}(B(0,1))}} \nonumber\\
&\lesssim  \sum_{l=1}^{k}{2^{\frac{l}{2}} \|f_{k,3}^{l}\|_{L^{2}(\mathbb{R}^{2})}} \nonumber\\
&\lesssim  k2^{\frac{k}{2}} \|f_{k,3}\|_{L^{2}(\mathbb{R}^{2})} \nonumber\\
&\le 2^{(-s+\frac{1}{2})k}k \|f\|_{H^{s}(\mathbb{R}^{2})}.
\end{align}
Inequalities (\ref{Eq2.3.8}), (\ref{Eq2.3.9}) and (\ref{Eq2.3.12}) imply when $k \gg 1$,
\begin{equation}\label{Eq2.3.13}
\biggl\|\mathop{sup}_{0<t<1} |e^{itP_{m}(D)}f_{k}|\biggl\|_{L^{2}(B(0,1))} \lesssim 2^{(-s+\frac{1}{2})k}k \|f\|_{H^{s}(\mathbb{R}^{2})},
\end{equation}
and then (\ref{Eq2.3.1}) follows.

(2) We can use the similar argument to give the proof of the positive result. Next we just show the counterexample for $1 < m <2$, $s < \frac{1}{2}$.

Define the subset of $\{\xi \in \mathbb{R}^{2}: |\xi| \sim R\}$ by
\[A_{R}=[R,R+1] \times [R, \frac{3R}{2}]\]
and define the function $f$ by
\[\hat{f}(\xi) = \chi_{A_{R}}(\xi).\]
It is obvious that
\begin{equation}\label{Eq2.3.14}
\|f\|_{L^{2}(\mathbb{R}^{2})}=R^{1/2}.
\end{equation}
By Taylor expansion, for each $\eta_{1} \in [0,1]$,
\[\frac{1}{m}|\eta_{1}+R|^{m} = \frac{1}{m}R^{m}+ R^{m-1} \eta_{1} + \frac{m-1}{2}|\theta \eta_{1}+R|^{m-2}\eta_{1}^{2},\hspace{0.5cm} \theta \in [0,1] \text{ depends on } \eta_{1}.\]
Hence by scaling and translating, we have
\begin{align}
|e^{itP_{m}(D)}f(x)| &= \biggl|\int_{R}^{\frac{3R}{2}}\int_{R}^{R+1}{e^{ix_{1}\xi_{1}+ix_{2}\xi_{2}+it(\xi_{1}\xi_{2}+\frac{1}{m}|\xi_{1}|^{m})}d\xi_{1}d\xi_{2}}\biggl| \nonumber\\
&=\frac{R}{2}\biggl|\int_{0}^{1}\int_{0}^{1}{e^{i(x_{1}+Rt)\eta_{1}+i\frac{R}{2}(x_{2}+Rt)\eta_{2}+i\frac{t}{2}(R\eta_{1}\eta_{2}+\frac{1}{m}|\eta_{1}+R|^{m})}d\eta_{1}d\eta_{2}}\biggl| \nonumber\\
&=\frac{R}{2}\biggl|\int_{0}^{1}\int_{0}^{1}{e^{i(x_{1}+Rt+R^{m-1}t)\eta_{1}+i\frac{R}{2}(x_{2}+Rt)\eta_{2}+i\frac{t}{2}(R\eta_{1}\eta_{2}+(m-1)|\theta \eta_{1}+R|^{m-2}\eta_{1}^{2})}d\eta_{1}d\eta_{2}}\biggl|.
\end{align}
Therefore, if $(x_{1}, x_{2}) \in [-1/1000,1/1000] \times [-1/2000,-1/1000]$, $t = -x_{2}/R + 1/R^{2}$ and $R$ is sufficiently large, then the abstract value of the phase function
\[|(x_{1}+Rt+R^{m-1}t)\eta_{1}+\frac{R}{2}(x_{2}+Rt)\eta_{2}+\frac{t}{2}(R\eta_{1}\eta_{2}+(m-1)|\theta \eta_{1}+R|^{m-2}\eta_{1}^{2}| \lesssim \frac{1}{1000}.\]
it follows that if $(x_{1}, x_{2}) \in [-1/1000,1/1000] \times [-1/2000,-1/1000]$, $t = -x_{2}/R + 1/R^{2}$ and $R$ is sufficiently large
\[|e^{itP_{m}(D)}f(x)| \gtrsim R.\]
Hence
\begin{equation}\label{Eq2.3.15}
 \big\|\mathop{sup}_{0<t<1} |e^{itP_{m}(D)}f|\big\|_{L^{1}(B(0,1))} \gtrsim R.
\end{equation}
Inequalities (\ref{Eq2.3.14}) and (\ref{Eq2.3.15}) imply (\ref{Eq2.3.2}).

%%%%%%%%%%%%%%%%%%%%%%%%%%%%%%%%%%%%%%%%%%%%%%%%%%%%%%%%%%%%%%%%%%%%%%%%%%%%%%%%%%%%%%%%%%%%%%%%%%%%%%%%%%%%%%%%%%%%%%%%%%%%%%%
\section{Proof of Theorem \ref{theorem2.4}}\label{Section 2}

\textbf{Proof of Theorem \ref{theorem2.4}.}
It is sufficient to show that for some $q \ge 1$ and $\forall \epsilon >0$,  $\forall x_{0} \in \mathbb{R}^{2}$, $s_{1} = s_{0} + \epsilon$,
\begin{equation}\label{Eq2.4}
\biggl\|\mathop{sup}_{0<t<1} \frac{|e^{itP(D)}(f)-f|}{t^{\delta/m}}\biggl\|_{L^{q}(B(x_{0},1))} \lesssim \|f\|_{H^{s_{1}+ \delta}(\mathbb{R}^{n})}.
\end{equation}
By translation, (\ref{Eq2.4}) can be reduced to
\begin{equation}\label{Eq22}
\biggl\|\mathop{sup}_{0<t<1} \frac{|e^{itP(D)}(f)-f|}{t^{\delta/m}}\biggl\|_{L^{q}(B(0,1))} \lesssim \|f\|_{H^{s_{1}+ \delta}(\mathbb{R}^{n})}.
\end{equation}
Concretely, if (\ref{Eq22}) holds for all $f \in H^{s_{1}+ \delta}(\mathbb{R}^{n})$, take $f_{0}$,
\[\hat{f_{0}}(\xi) = e^{ix_{0} \cdot \xi}\hat{f}(\xi)\]
and insert $f_{0}$ into (\ref{Eq22}). Then (\ref{Eq2.4}) follows from simple computation.

Next we show (\ref{Eq2.4}) implies (\ref{Eq2.4.3}). In fact, if (\ref{Eq2.4}) holds, then fix $\lambda >0$, for any $\epsilon >0$, choose $g \in C_{c}^{\infty}(\mathbb{R}^{n})$ such that
\begin{equation}
\|f-g\|_{H^{s_{1}+ \delta}(\mathbb{R}^{n})} \le  \frac{\lambda \epsilon^{1/q}}{2},
\end{equation}
it follows
\begin{align}
&\biggl|\biggl\{{ x \in B(x_{0},1): \mathop{sup}_{0<t<1} \frac{|e^{itP(D)}(f-g)-(f-g)|}{t^{\delta/m}}} > \frac{\lambda}{2}\biggl\}\biggl|  \nonumber\\
&\le \frac{2^{q}}{\lambda^{q}} \biggl\|\mathop{sup}_{0<t<1} \frac{|e^{itP(D)}(f-g)-(f-g)|}{t^{\delta/m}}\biggl\|_{L^{q}(B(x_{0},1))}^{q} \nonumber\\
&\lesssim \frac{2^{q}}{\lambda^{q}}\|f-g\|_{H^{s_{1}+ \delta}(\mathbb{R}^{n})}^{q} \nonumber\\
&\le \epsilon,
\end{align}
and
\begin{align}
\frac{|e^{itP(D)}(g)(x)-g(x)|}{t^{\delta/m}} &\le t^{1-\frac{\delta}{m}} \int_{\mathbb{R}^{n}}{|P(\xi)\hat{g}(\xi)|d\xi}\rightarrow 0,\hspace{0.5cm} \textmd{if } \hspace{0.2cm}t \rightarrow 0^{+}
\end{align}
uniformly for $x \in B(x_{0},1)$. Then we have
\begin{align}
\biggl|\biggl\{{ x \in B(x_{0},1): \mathop{lim sup}_{t \rightarrow 0^{+}} \frac{|e^{itP(D)}(f)(x)-(f)(x)|}{t^{\delta/m}}} > \lambda\biggl\}\biggl| \le \epsilon,
\end{align}
which  implies (\ref{Eq2.4.3}) for $f \in H^{s_{1} +\delta}(\mathbb{R}^{n})$ and $x \in B(x_{0}, 1)$. By the arbitrariness of $\epsilon$ and $x_0$, in fact we can get (\ref{Eq2.4.3}) for all $f \in H^{s +\delta}(\mathbb{R}^{n})$, $s > s_{0}$ and $x \in \mathbb{R}^{n}$. Next we will prove (\ref{Eq22}) for $q=min\{p,2\}$.

In order to prove (\ref{Eq22}), we decompose $f$ as
\[f=\sum_{k=0}^{\infty}{f_{k}},\]
where supp$ \hat{f_{0}} \subset B(0,1)$, supp$ \hat{f_{k}} \subset \{\xi: |\xi| \sim 2^{k}\}$, $k \ge 1$. It follows that
\begin{equation}\label{Eq2.9}
\biggl\|\mathop{sup}_{0<t<1} \frac{|e^{itP(D)}(f)-f|}{t^{\delta/m}}\biggl\|_{L^{q}(B(0,1))} \le \sum_{k=0}^{\infty}{\biggl\|\mathop{sup}_{0<t<1} \frac{|e^{itP(D)}(f_{k})-f_{k}|}{t^{\delta/m}}\biggl\|_{L^{q}(B(0,1))}}.
\end{equation}
By Taylor's formula, for each $k$,
\begin{equation}\label{Eq2.10}
\frac{|e^{itP(D)}(f_{k})-f_{k}|}{t^{\delta/m}} \le \sum_{j=1}^{\infty}{\frac{t^{j-\delta/m}}{j!} \biggl|\int_{\mathbb{R}^{n}}{e^{ix\cdot\xi}P(\xi)^{j}\hat{f_{k}}}(\xi)d\xi\biggl|  }.
\end{equation}

For $k \lesssim 1$, because (\ref{Eq2.10}) and $P(\xi)$ is continuous,
\begin{align}\label{Eq2.11}
\biggl\|\mathop{sup}_{0<t<1} \frac{|e^{itP(D)}(f_{k})-f_{k}|}{t^{\delta/m}}\biggl\|_{L^{q}(B(0,1))} &\le \sum_{j=1}^{\infty}{\frac{1}{j!} \biggl\|\int_{\mathbb{R}^{n}}{e^{ix\cdot\xi}P(\xi)^{j}\hat{f_{k}}}(\xi)d\xi\biggl\|_{L^{q}(B(0,1))}} \nonumber\\
&\le \sum_{j=1}^{\infty}{\frac{1}{j!} \biggl\|\int_{\mathbb{R}^{n}}{e^{ix\cdot\xi}P(\xi)^{j}\hat{f_{k}}}(\xi)d\xi\biggl\|_{L^{2}(B(0,1))}} \nonumber\\
&\le \sum_{j=1}^{\infty}{\frac{1}{j!} \|P(\xi)^{j}\hat{f_{k}}}(\xi)\|_{L^{2}(\mathbb{R}^{n})} \nonumber\\
&\lesssim \|f\|_{H^{s_{1}+ \delta}(\mathbb{R}^{n})}.
\end{align}

For $k \gg 1$,
\begin{align}
\biggl\|\mathop{sup}_{0<t<1} \frac{|e^{itP(D)}(f_{k})-f_{k}|}{t^{\delta/m}}\biggl\|_{L^{q}(B(0,1))} &\le \biggl\|\mathop{sup}_{0<t<2^{-mk}} \frac{|e^{itP(D)}(f_{k})-f_{k}|}{t^{\delta/m}}\biggl\|_{L^{q}(B(0,1))} \nonumber\\
& \:\ + \biggl\|\mathop{sup}_{2^{-mk} \le t < 1} \frac{|e^{itP(D)}(f_{k})-f_{k}|}{t^{\delta/m}}\biggl\|_{L^{q}(B(0,1))}.
\end{align}
Inequalities (\ref{Eq2.10}) and (\ref{Eq2.4.1}) imply
\begin{align}\label{Eq2.13}
 \biggl\|\mathop{sup}_{0<t<2^{-mk}} \frac{|e^{itP(D)}(f_{k})-f_{k}|}{t^{\delta/m}}\biggl\|_{L^{q}(B(0,1))} &\le \sum_{j=1}^{\infty}{\frac{2^{-mkj+\delta k}}{j!} \biggl\|\int_{\mathbb{R}^{n}}{e^{ix\cdot\xi}P(\xi)^{j}\hat{f_{k}}}(\xi)d\xi\biggl\|_{L^{q}(B(0,1))}} \nonumber\\
&\le \sum_{j=1}^{\infty}{\frac{2^{-mkj+\delta k}}{j!} \biggl\|\int_{\mathbb{R}^{n}}{e^{ix\cdot\xi}P(\xi)^{j}\hat{f_{k}}}(\xi)d\xi \biggl\|_{L^{2}(B(0,1))}} \nonumber\\
&\le \sum_{j=1}^{\infty}{\frac{2^{-mkj+\delta k}}{j!} \|P(\xi)^{j}\hat{f_{k}}}(\xi)\|_{L^{2}(\mathbb{R}^{2})} \nonumber\\
&\le \sum_{j=1}^{\infty}{\frac{2^{-mkj+\delta k}2^{mkj}}{j!} \|\hat{f_{k}}}(\xi)\|_{L^{2}(\mathbb{R}^{2})} \nonumber\\
&\lesssim 2^{-s_{1}k}\|f\|_{H^{s_{1}+ \delta}(\mathbb{R}^{n})}.
\end{align}
From (\ref{Eq2.4.2}) we have,
\begin{equation}
\biggl\|\mathop{sup}_{2^{-mk} \le t<1} |e^{itP(D)}f_{k}|\biggl\|_{L^{p}(B(0,1))} \lesssim 2^{(s_{0}+\frac{\epsilon}{2})k}\|f_{k}\|_{L^{2}(\mathbb{R}^{n})},
\end{equation}
hence,
\begin{align}\label{Eq2.15}
 \biggl\|\mathop{sup}_{2^{-mk} \le t < 1} \frac{|e^{itP(D)}(f_{k})-f_{k}|}{t^{\delta/m}}\biggl\|_{L^{q}(B(0,1))} &\le 2^{\delta k} \biggl\|\mathop{sup}_{2^{-mk} \le t < 1} |e^{itP(D)}(f_{k})-f_{k}|\biggl\|_{L^{q}(B(0,1))} \nonumber\\
&\le 2^{\delta k} \biggl\{\biggl\|\mathop{sup}_{2^{-mk} \le t < 1} |e^{itP(D)}(f_{k})|\biggl\|_{L^{q}(B(0,1))} + \|f_{k}|\|_{L^{q}(B(0,1))} \biggl\}\nonumber\\
&\lesssim 2^{\delta k} \biggl\{\biggl\|\mathop{sup}_{2^{-mk} \le t < 1} |e^{itP(D)}(f_{k})|\biggl\|_{L^{p}(B(0,1))} + \|f_{k}|\|_{L^{2}(B(0,1))} \biggl\} \nonumber\\
&\lesssim 2^{\delta k}2^{(s_{0}+\frac{\epsilon}{2})k}\|f\|_{L^{2}(\mathbb{R}^{n})} \nonumber\\
&\lesssim 2^{-\frac{\epsilon k}{2}}\|f\|_{H^{s_{1}+ \delta}(\mathbb{R}^{n})}.
\end{align}
Inequalities (\ref{Eq2.13}) and (\ref{Eq2.15}) yield for $k \gg 1$,
\begin{align}\label{Eq2.16}
\biggl\|\mathop{sup}_{0<t<1} \frac{|e^{itP(D)}(f_{k})-f_{k}|}{t^{\delta/m}}\biggl\|_{L^{q}(B(0,1))} &\lesssim 2^{-\frac{\epsilon k}{2}}\|f\|_{H^{s_{1}+ \delta}(\mathbb{R}^{n})}.
\end{align}

It is clear that (\ref{Eq22}) follows from (\ref{Eq2.9}), (\ref{Eq2.11}) and (\ref{Eq2.16}).

\end{document}